\documentclass{amsart}
\usepackage{amsfonts,amssymb,amscd,amsmath,enumerate,
verbatim,calc}
\numberwithin{equation}{section}
\newtheorem{thm}{Theorem}[section]
\newtheorem{cor}[thm]{Corollary}

\newtheorem{prop}[thm]{Proposition}
\newtheorem{defn}[thm]{Definition}

\newtheorem{ques}[thm]{Question}
\newtheorem{con}[thm]{Conjecture}

\newcommand{\Hom}{\mbox{Hom}\,}
\newcommand{\Ext}{\mbox{Ext}\,}
\newcommand{\Tor}{\mbox{Tor}\,}

\newcommand{\Ass}{\mbox{Ass}\,}

\newcommand{\Supp}{\mbox{Supp}\,}

\newcommand{\gr}{\mbox{grade}\,}

\newcommand{\cd}{\mbox{cd}\,}

\newcommand{\im}{\mbox{Im}\,}

\newcommand{\lo}{\longrightarrow}

\newcommand{\ara}{\mbox{ara}}
\newcommand{\fa}{\mathfrak{a}}
\newcommand{\fb}{\mathfrak{b}}

\newcommand{\fm}{\mathfrak{m}}
\newcommand{\fp}{\mathfrak{p}}

\newcommand{\fc}{\mathfrak{c}}

\begin{document}

\title[extension functors of local cohomology modules]
 {\large extension functors of local cohomology modules}

\bibliographystyle{amsplain}

    \author[M. Aghapournahr]{M. Aghapournahr$^{1}$}
\address{$^{1}$ Arak University, Beheshti St, P.O. Box: 879, Arak, Iran.}
\email{m-aghapour@araku.ac.ir}

    \author[A. J. Taherizadeh]{A. J. Taherizadeh$^{2}$}
\address{$^{2, 3}$ Faculty of Mathematical Sciences and Computer,
Tarbiat Moallem University, Tehran, Iran.}
\email{taheri@saba.tmu.ac.ir}

   \author[A. Vahidi]{Alireza Vahidi$^{3}$}
\address{$^{3}$ Payame Noor University (PNU), Iran.}
\email{vahidi.ar@gmail.com}

\keywords{Local cohomology modules, Serre subcategories, Cofinite
modules.}

\subjclass[2000]{13D45, 13D07.}


\begin{abstract}
Let $R$ be a commutative Noetherian ring with non-zero identity,
$\fa$ an ideal of $R$, and $X$ an $R$--module. In this paper, for
fixed integers $s, t$ and a finite $\fa$--torsion $R$--module $N$,
we first study the membership of $\Ext^{s+t}_{R}(N, X)$ and
$\Ext^{s}_{R}(N, H^{t}_{\fa}(X))$ in Serre subcategories of the
category of $R$--modules. Then we present some conditions which
ensure the existence of an isomorphism between them. Finally, we
introduce the concept of Serre cofiniteness as a generalization of
cofiniteness and study this property for certain local cohomology
modules.
\end{abstract}

\maketitle

\section{Introduction}

Throughout $R$ will denote a commutative Noetherian ring with
non-zero identity and $\fa$ an ideal of $R$. Also $N$ will be a
finite $\fa$--torsion module and $X$ an $R$--module. For unexplained
terminology from homological and commutative algebra we refer the
reader to \cite {BSh} and \cite {BH}.

The following conjecture was made by Grothendieck in \cite {G}.

\begin{con} For any ideal $\fa$ and finite $R$--module $X$, the
module $\emph{\Hom}_{R}(R/\fa, H^{n}_{\fa}(X))$ is finite for all
$n\geq 0$.
\end{con}

This conjecture is false in general as shown by Hartshorne in
\cite{Ha2}. However, he defined an $R$--module $X$ to be {\it
$\fa$--cofinite} if $\Supp_{R}(X)\subseteq V(\fa)$ and
$\Ext^{i}_{R}(R/\fa, X)$ is finite for each $i$, and he asked the
following question.

\begin{ques}
If $\fa$ is an ideal of $R$ and $X$ is a finite $R$--module when is
$\emph{\Ext}^{i}_{R}(R/\fa, H^{j}_{\fa}(X))$ finite for every $i$
and $j$?
\end{ques}

There are some attempts to show that under some conditions, for
fixed integers $s$ and $t$, the $R$--module $\Ext^{s}_{R}(R/\fa,
H^{t}_{\fa}(X))$ is finite, for example see \cite [Theorem 3.3]
{AKS}, \cite [Theorems A and B] {DY1}, \cite [Theorem 6.3.9] {DY2}
and \cite [Theorem 3.3] {KH}.

Recently, the first author and Melkersson in \cite{AM1} and
\cite{AM2}, and Asgharzadeh and Tousi in \cite{AT} approached the
study of local cohomology modules by means of Serre subcategories
and it is noteworthy that their approach enables us to deal with
several important problems on local cohomology modules
comprehensively. For more information, we refer the reader to \cite
{Hu} to see a survey of some important problems on finiteness,
vanishing, Artinianness, and finiteness of associated primes of
local cohomology modules.

In this paper, we study some properties of extension functors of
local cohomology modules by using Serre classes. Recall that a class
of $R$--modules is a {\it Serre subcategory} of the category of
$R$--modules when it is closed under taking submodules, quotients
and extensions. Always, $\mathcal{S}$ stands for a Serre subcategory
of the category of $R$--modules.

The crucial points of Section 2 are Theorems \ref {2.1} and \ref
{2.3} which show that when $R$--modules $\Ext^{s+t}_{R}(N, X)$ and
$\Ext^{s}_{R}(N, H^{t}_{\fa}(X))$ belong to $\mathcal{S}$. These two
theorems, which are frequently used through the paper, enable us to
demonstrate some new facts and improve some older facts about the
extension functors of local cohomology modules. We find the weakest
possible conditions for finiteness of associated primes of local
cohomology modules and, improve and give a new proof for \cite
[Theorem 3.3] {KH} in Corollaries \ref {2.5} and \ref {2.7}. The
relation between $R$--modules $\Ext^{s+t}_{R}(N, X)$ and
$\Ext^{s}_{R}(N, H^{t}_{\fa}(X))$ to be in a Serre subcategory of
the category of $R$--modules is shown in Corollary \ref {2.8}.

In Section 3, we first introduce the class of Melkersson subcategory
as a special case of Serre classes and next investigate the
extension functors of local cohomology modules in these
subcategories. In Propositions \ref {3.2}, \ref {3.3} and \ref
{3.4}, we give new proofs for \cite [Theorems 2.9 and 2.13] {AM1}
and study the membership of the local cohomology modules of an
$R$--module $X$ with respect to different ideals in Melkersson
subcategories. Our main result in this section is Theorem \ref {3.5}
which provides an isomorphism between the $R$--modules
$\Ext^{s+t}_{R}(N, X)$ and $\Ext^{s}_{R}(N, H^{t}_{\fa}(X))$.
Corollaries \ref {3.6} through \ref {3.9} are some applications of
this theorem.

In Section 4, we present a generalization of the concept of
cofiniteness with respect to an ideal to Serre subcategories of the
category of $R$--modules. Theorems \ref {4.2}, \ref {4.4} and \ref
{4.6} generalize \cite [Proposition 2.5] {MV}, \cite [Proposition
3.11] {Mel}, \cite [Theorem 3.1] {DM2}, \cite [Theorems A and B]
{DY1} and \cite [Corollary 2.7]{DM1}. The Change of ring principle
for Serre cofiniteness is presented in Theorem \ref {4.8}. We also
give a proposition about $\fa$--cofinite minimax local cohomology
modules in Proposition \ref {4.10}. Corollaries \ref {4.11} and \ref
{4.12} are immediate results of this proposition where Corollary
\ref {4.11} improves \cite [Theorem 2.3] {BN}.


\section{Local cohomology modules and Serre subcategories}

Let $\fa$ be an ideal of $R$, $N$ a finite $\fa$--torsion module and
$s, t$ non-negative integers. In this section, we present sufficient
conditions which convince us the $R$--modules $\Ext^{t}_{R}(N ,X)$
and $\Ext^{s}_{R}(N, H^{t}_{\fa}(X))$ are in a Serre subcategory of
the category of $R$--modules. Even though we can provide elementary
proofs by using induction for our main theorems, for shortening the
proofs we use spectral sequences argument.



\begin{thm} \label {2.1} Let $X$ be an $R$--module and $t$ be a non-negative integer
such that $\emph{\Ext}^{t-r}_{R}(N, H^{r}_{\fa}(X))$ is in
$\mathcal{S}$ for all $r$, $0\leq r\leq t.$ Then
$\emph{\Ext}^{t}_{R}(N, X)$ is in $\mathcal{S}.$
\end{thm}

\begin{proof} By \cite [Theorem 11.38]{Rot}, there is a Grothendieck spectral sequence
$$E^{p, q}_{2}:= \Ext^{p}_{R}(N, H^{q}_{\frak{a}}(X))
_{\stackrel{\Longrightarrow}{p}} \Ext^{p +q}_{R}(N, X).$$ For all
$r$, $0\leq r\leq t$, we have $E_{\infty}^{t-r, r}= E_{t+2}^{t-r,
r}$ since $E_{i}^{t-r-i, r+i-1}= 0= E_{i}^{t-r+i, r+1-i}$ for all
$i\geq t+2$; so that $E_{\infty}^{t-r, r}$ is in $\mathcal{S}$ from
the fact that $E_{t+2}^{t-r, r}$ is a subquotient of $E_{2}^{t-r,
r}$ which is in $\mathcal{S}$ by assumption. There exists a finite
filtration
$$0= \phi^{t+1}H^{t}\subseteq \phi^{t}H^{t}\subseteq \cdots
 \subseteq \phi^{1}H^{t}\subseteq \phi^{0}H^{t}= \Ext^{t}_{R}(N,X)$$
 such that $E_{\infty}^{t-r, r}= \phi^{t-r}H^{t}/\phi^{t-r+1}H^{t}$ for
 all $r$, $0\leq r\leq t$. Now the exact sequences
 $$0\longrightarrow \phi^{t-r+1}H^{t}\longrightarrow \phi^{t-r}H^{t}\longrightarrow
E_{\infty}^{t-r, r}\longrightarrow 0,$$ for
 all $r$, $0\leq r\leq t,$ yield the assertion.
\end{proof}



Recall that, an $R$--module $X$ is said to be {\it weakly Laskerian}
if the set of associated primes of any quotient module of $X$ is
finite (see \cite [Definition 2.1]{DM1}). Also, we say that $X$ is
{\it $\fa$--weakly cofinite} if $\Supp_{R}(X)\subseteq V(\fa)$ and
$\Ext^{i}_{R}(R/\fa, X)$ is weakly Laskerian for all $i\geq 0$ (see
\cite [Definition 2.4]{DM2}). We denote the category of $R$--modules
(resp. the category of finite $R$--modules, the category of weakly
Laskerian $R$--modules) by $\mathcal{C}(R)$ (resp.
$\mathcal{C}_{f.g}(R)$, $\mathcal{C}_{w.l}(R)$).

\begin{cor} \label {2.2} {\rm (cf.} \cite [Theorem 6.3.9(i)]{DY2}{\rm )}
Let $X$ be an $R$--module and $n$ be a non-negative integer such
that for all $r$, $0\leq r\leq n,$ $\emph{\Ext}^{n-r}_{R}(N,
H^{r}_{\fa}(X))$ is weakly Laskerian {\rm(}resp. finite{\rm)}. Then
$\emph{\Ext}^{n}_{R}(N, X)$ is weakly Laskerian {\rm(}resp.
finite{\rm)} and so $\emph{\Ass}_{R}(Ext^{n}_{R}(N ,X))$ is finite.
\end{cor}



The next theorem is related to the $R$--module $\Ext^{s}_{R}(N,
H^{t}_{\fa}(X))$ to be in a Serre subcategory of the category of
$R$--modules.

\begin{thm} \label {2.3} Let $X$ be an $R$--module and $s, t$ be non-negative integers such that
               \begin{itemize}
                   \item[(i)]{$\emph{\Ext}^{s+ t}_{R}(N, X)$ is in $ \mathcal{S}$,}
                   \item[(ii)]{$\emph{\Ext}^{s+ t+ 1- i}_{R}(N, H^{i}_{\fa}(X))$ is in $\mathcal{S}$ for all $i$, $0\leq i< t,$ and}
                   \item[(iii)]{$\emph{\Ext}^{s+ t- 1- i}_{R}(N, H^{i}_{\fa}(X))$ is in $\mathcal{S}$ for all $i$, $t+ 1\leq i< s+ t.$}
               \end{itemize}
Then $\emph{\Ext}^{s}_{R}(N, H^{t}_{\fa}(X))$ is in $\mathcal{S}$.
\end{thm}

\begin{proof} Consider the Grothendieck spectral sequence
$$E^{p, q}_{2}:= \Ext^{p}_{R}(N, H^{q}_{\frak{a}}(X))
_{\stackrel{\Longrightarrow}{p}} \Ext^{p+ q}_{R}(N, X).$$ For all
$r\geq 2$, let $Z_{r}^{s, t}= \ker(E_{r}^{s, t}\longrightarrow
E_{r}^{s+r, t+1-r})$ and $B_{r}^{s, t}= \im(E_{r}^{s-r,
t+r-1}\longrightarrow E_{r}^{s, t})$. We have the exact sequences:
$$0\longrightarrow Z_{r}^{s, t}\longrightarrow E_{r}^{s, t}\longrightarrow E_{r}^{s, t}/Z_{r}^{s, t}\longrightarrow 0$$
and
$$0\longrightarrow B_{r}^{s ,t}\longrightarrow Z_{r}^{s, t}\longrightarrow E_{r+1}^{s, t}\longrightarrow 0.$$
Since, by assumptions (ii) and (iii), $E_{2}^{s+r, t+1-r}$ and
$E_{2}^{s-r, t+r-1}$ are in $\mathcal{S}$, $E_{r}^{s+r, t+1-r}$ and
$E_{r}^{s-r, t+r-1}$ are also in $\mathcal{S}$, and so $E_{r}^{s,
t}/Z_{r}^{s, t}$ and $B_{r}^{s, t}$ are in $\mathcal{S}$. It shows
that $E_{r}^{s, t}$ is in $\mathcal{S}$ whenever $E_{r+1}^{s ,t}$ is
in $\mathcal{S}$.

We have $E_{r}^{s-r, t+r-1}= 0= E_{r}^{s+r, t+1-r}$ for all $r$,
$r\geq t+s+2$. Therefore we obtain $E_{t+s+2}^{s, t}= E_{\infty}^{s,
t}$. To complete the proof, it is enough to show that
$E_{\infty}^{s, t}$ is in $\mathcal{S}$. There exists a finite
filtration
$$0= \phi^{s+t+1}H^{s+t}\subseteq \phi^{s+t}H^{s+t}\subseteq \cdots
 \subseteq \phi^{1}H^{s+t}\subseteq \phi^{0}H^{s+t}= \Ext^{s+t}_{R}(N, X)$$
such that $E_{\infty}^{s+t-j ,j}=
\phi^{s+t-j}H^{s+t}/\phi^{s+t-j+1}H^{s+t}$ for all $j$, $0\leq j\leq
s+t$. Since $\Ext^{s+t}_{R}(N, X)$ is in $\mathcal{S}$,
$\phi^{s}H^{s+t}$ is in $\mathcal{S}$ and so $E_{\infty}^{s, t}=
\phi^{s}H^{s+t}/\phi^{s+ 1}H^{s+t}$ is in $\mathcal{S}$ as we
desired.
\end{proof}



\begin{cor} \label {2.4} {\rm(cf.} \cite [Theorem 2.2]{AT}{\rm)}
Suppose that $X$ is an $R$--module and $n$ is a non-negative integer
such that
               \begin{itemize}
                   \item[(i)]{$\emph{\Ext}^{n}_{R}(N, X)$ is in $\mathcal{S}$, and}
                   \item[(ii)]{$\emph{\Ext}^{n+1-i}_{R}(N, H^{i}_{\fa}(X))$ is in $\mathcal{S}$ for all $i$, $0\leq i< n.$}
               \end{itemize}
Then $\emph{\Hom}_{R}(N, H^{n}_{\fa}(X))$ is in $\mathcal{S}.$
\end{cor}

\begin{proof}
Apply Theorem \ref {2.3} with $s=0$ and $t= n$.
\end{proof}



We can deduce from the above corollary the main results of \cite
[Theorem B]{KS}, \cite [Theorem 2.2]{BL}, \cite [Theorem 5.6]{N},
\cite [Corollary 2.7]{DM1}, \cite [Theorem 6.3.9(ii)]{DY2}, \cite
[Theorem 2.3]{BSS}, \cite [Corollary 3.2]{DN}, \cite [Corollary
2.3]{BSY} and \cite [Lemma 2.2]{BN} concerning the finiteness of
associated primes of local cohomology modules. We just state the
weakest possible conditions which yield the finiteness of associated
primes of local cohomology modules in the next corollary.

\begin{cor} \label {2.5} Suppose that $X$ is an $R$--module and $n$ is a non-negative integer
such that
               \begin{itemize}
                   \item[(i)]{$\emph{\Ext}^{n}_{R}(R/\fa, X)$ is weakly Laskerian, and}
                   \item[(ii)]{$\emph{\Ext}^{n+1-i}_{R}(R/\fa, H^{i}_{\fa}(X))$ is weakly Laskerian  for all $i$, $0\leq i< n.$}
               \end{itemize}
Then $\emph{\Hom}_{R}(R/\fa, H^{n}_{\fa}(X))$ is weakly Laskerian,
and so $\emph{\Ass}_{R}(H^{n}_{\fa}(X))$  is finite.
\end{cor}

\begin{proof}
Apply Corollary \ref {2.4} with $N= R/\fa$ and $\mathcal{S}=
\mathcal{C}_{w.l}(R)$, and note that we have the equality
$\Ass_{R}(\Hom_{R}(R/\fa, H^{n}_{\fa}(X)))= V(\fa)\cap
\Ass_{R}(H^{n}_{\fa}(X))= \Ass_{R}(H^{n}_{\fa}(X))$.
\end{proof}



It is easy to see that if $R$ is a local ring and $\mathcal{S}$ is a
non-zero Serre subcategory of the category of $R$--modules, then
every $R$--module with finite length belongs to $\mathcal{S}$.

\begin{cor} \label {2.6} {\rm(cf.} \cite [Theorem 2.12]{AT}{\rm)}
Let $R$ be a local ring with maximal ideal $\fm$ and $X$ be an
$R$--module. Assume also that $\mathcal{S}$ is a non-zero Serre
subcategory of $\mathcal{C}(R)$ and $n$ is a non-negative integer
such that
               \begin{itemize}
                   \item[(i)]{$\emph{\Ext}^{n}_{R}(R/\fm, X)$ is finite, and}
                   \item[(ii)]{$\emph{\Ext}^{n+1-i}_{R}(R/\fm ,H^{i}_{\fa}(X))$ is in $\mathcal{S}$ for all $i$, $0\leq i< n.$}
               \end{itemize}
Then $\emph{\Hom}_{R}(R/\fm, H^{n}_{\fa}(X))$ is in $\mathcal{S}$.
\end{cor}

\begin{proof} Since $\mathcal{S}\neq 0$, $\Ext^{n}_{R}(R/\fm, X)$ is in
$\mathcal{S}$. Now, the assertion follows from Corollary \ref {2.4}.
\end{proof}



Khashayarmanesh, in \cite [Theorem 3.3]{KH}, by using the concept of
$\fa$--filter regular sequence, proved the following corollary with
stronger assumptions. His assumptions were $X$ is a finite
$R$--module with finite Krull dimension and $N= R/\fb,$ where $\fb$
is an ideal of $R$ contains $\fa$, while it is a simple conclusion
of Theorem \ref {2.3} for an arbitrary $R$--module $X$ and a finite
$\fa$--torsion module $N.$

\begin{cor} \label {2.7} {\rm(cf.} \cite [Theorem 3.3]{KH}{\rm)}
Suppose that $X$ is an $R$--module and $s, t$ are non-negative
integers such that
               \begin{itemize}
                   \item[(i)]{$\emph{\Ext}^{s+ t}_{R}(N, X)$ is finite,}
                   \item[(ii)]{$\emph{\Ext}^{s+ t+ 1- i}_{R}(N, H^{i}_{\fa}(X))$ is finite for all $i$, $0\leq i< t,$ and}
                   \item[(iii)]{$\emph{\Ext}^{s+ t- 1- i}_{R}(N, H^{i}_{\fa}(X))$ is finite for all $i$, $t+ 1\leq i< s+ t.$}
               \end{itemize}
Then $\emph{\Ext}^{s}_{R}(N, H^{t}_{\fa}(X))$ is finite.
\end{cor}

\begin{proof}
Apply Theorem \ref {2.3} for $\mathcal{S}= \mathcal{C}_{f.g}(R)$.
\end{proof}



Theorem \ref {2.1} in conjunction with Theorem \ref {2.3} arise the
following corollary.

\begin{cor} \label {2.8} Let $X$ be an $R$--module and $n, m$ be  non-negative
integers such that $n\leq m$. Assume also that $H^{i}_{\fa}(X)$ is
in $\mathcal{S}$ for all $i$, $i\neq n$ {\rm(}resp. $0\leq i\leq
n-1$ or $n+1\leq i\leq m${\rm)}. Then, for all $i$, $i\geq 0$
{\rm(}resp. $0\leq i\leq m-n${\rm)}, $\emph{\Ext}^{i}_{R}(N,
H^{n}_{\frak{a}}(X))$ is in $\mathcal{S}$ if and only if
$\emph{\Ext}^{i+ n}_{R}(N, X)$ is in $\mathcal{S}$.
\end{cor}



In the course of the remaining parts of the paper by
$\cd_{\mathcal{S}}(\fa, X)$ ($\mathcal{S}$--cohomological dimension
of $X$ with respect to $\fa$) we mean the largest integer $i$ in which
$H^{i}_{\fa}(X)$ is not in $\mathcal{S}$ (see \cite [Definition 3.4]
{AT} or  \cite [Definition 3.5] {AM1}). Note that when $\mathcal{S}=
0$, then $\cd_{\mathcal{S}}(\fa, X)= \cd(\fa, X)$ as in \cite{Ha1}.

\begin{cor} \label {2.9} Let $X$ be an $R$--module and $n$ be a non-negative integer.
Then the following statements hold true.
               \begin{enumerate}
                   \item[\emph{(i)}]{If $\emph{\cd}_{\mathcal{S}}(\fa, X)= 0$,
                   then $\emph{\Ext}_{R}^{n}(N, \Gamma_{\fa}(X))$ is in $\mathcal{S}$ if and only if $\emph{\Ext}_{R}^{n}(N, X)$ is in $\mathcal{S}$.}
                   \item[\emph{(ii)}]{If $\emph{\cd}_{\mathcal{S}}(\fa, X)= 1$,
                   then $\emph{\Ext}_{R}^{n}(N, H^{1}_{\fa}(X))$ is in $\mathcal{S}$ if and only if $\emph{\Ext}_{R}^{n+ 1}(N, X/\Gamma_{\fa}(X))$ is in $\mathcal{S}$.}
                   \item[\emph{(iii)}]{If $\emph{\cd}_{\mathcal{S}}(\fa, X)= 2$,
                   then $\emph{\Ext}_{R}^{n}(N, H^{2}_{\fa}(X))$ is in $\mathcal{S}$ if and only if $\emph{\Ext}_{R}^{n +2}(N, D_{\fa}(X))$ is in $\mathcal{S}$.}
               \end{enumerate}
\end{cor}

\begin{proof}(i) This is clear from Corollary \ref {2.8}.

(ii) For all $i\neq 1$, $H^{i}_{\fa}(X/\Gamma_{\fa}(X))$ is in
$\mathcal{S}$ by assumption. Now, the assertion follows from
Corollary \ref {2.8}.

(iii) By \cite [Corollary 2.2.8]{BSh}, $H^{i}_{\fa}(D_{\fa}(X))$ is
in $\mathcal{S}$ for all $i\neq 2$. Again, use Corollary \ref {2.8}.
\end{proof}




\section{Special Serre subcategories}

In this section, we study the extension functors of local cohomology
modules in some special Serre subcategories of the category of
$R$--modules. We begin with a definition.



\begin{defn}\label{3.1}  {\rm(}see \cite [Definition 2.1]{AM1}{\rm)}
Let $\mathcal{M}$ be a Serre subcategory of the category of
$R$--modules. We say that $\mathcal{M}$ is a {\it Melkersson
subcategory with respect to the ideal $\fa$} if for any
$\fa$--torsion $R$--module $X$, $0:_{X}\fa$ is in $\mathcal{M}$
implies that $X$ is in $\mathcal{M}$. $\mathcal{M}$ is called {\it
Melkersson subcategory} when it is a Melkersson subcategory with
respect to all ideals of $R$.
\end{defn}

In honor of Melkersson who proved this property for Artinian
category (see \cite [Theorem 7.1.2]{BSh}) and Artinian
$\fa$--cofinite category (see \cite [Proposition 4.1]{Mel}), we
named the above subcategory as Melkersson subcategory. To see some
examples of Melkersson subcategories, we refer the reader to \cite
[Examples 2.4 and 2.5]{AM1}.\\

The next two propositions show that how properties of Melkersson
subcategories behave similarly at the initial points of $\Ext$ and
local cohomology modules. These propositions give new proofs for
\cite [Theorems 2.9 and 2.13]{AM1} based on Theorems \ref{2.1} and
\ref {2.3}.



\begin{prop} \label {3.2} {\rm(}see  \cite [Theorem 2.13]{AM1}{\rm)}
Let $X$ be an $R$--module, $\mathcal{M}$ be a Melkersson subcategory
with respect to the ideal $\fa$, and $n$ be a non-negative integer
such that $\emph{\Ext}^{j-i}_{R}(R/\fa, H^{i}_{\fa}(X))$ is in
$\mathcal{M}$ for all $i, j$ with $0\leq i\leq n-1$ and $j= n, n+1.$
Then the following statements are equivalent.
      \begin{itemize}
           \item[(i)]{$\emph{\Ext}^{n}_{R}(R/\fa, X)$ is in $\mathcal{M}$.}
           \item[(ii)]{$H^{n}_{\fa}(X)$ is in $\mathcal{M}$.}
      \end{itemize}
\end{prop}

\begin{proof}(i) $\Rightarrow$ (ii). Apply Theorem \ref {2.3} with
$s= 0$ and $t= n$. It shows that $\Hom_{R}(R/\fa, H^{n}_{\fa}(X))$
is in $\mathcal{M}$. Thus $H^{n}_{\fa}(X)$ is in $\mathcal{M}$.

(ii) $\Rightarrow$ (i). Apply Theorem \ref {2.1} with $t= n.$
\end{proof}



\begin{prop} \label {3.3} {\rm(}see \cite [Theorem 2.9]{AM1}{\rm)}
Let $X$ be an $R$--module, $\mathcal{M}$ be a Melkersson subcategory
with respect to the ideal $\fa$, and $n$ be a non-negative integer.
Then the following statements are equivalent.
                         \begin{itemize}
                              \item[(i)]{$H^{i}_{\fa}(X)$ is in $\mathcal{M}$ for all $i$, $0\leq i\leq n$.}
                              \item[(ii)]{$\emph{\Ext}^{i}_{R}(R/\fa, X)$ is in $\mathcal{M}$ for all $i$, $0\leq i\leq n$.}
                         \end{itemize}
\end{prop}

\begin{proof}
(i) $\Rightarrow$ (ii). Let $0\leq t\leq n.$ Since $H^{r}_{\fa}(X)$
is in $\mathcal{M}$ for all  $r$, $0\leq r\leq t,$
$\Ext^{t-r}_{R}(R/\fa, H^{r}_{\fa}(X))$ is in $\mathcal{M}$ for all
$r$, $0\leq r\leq t.$ Hence $\Ext^{t}_{R}(R/\fa, X)$ is in
$\mathcal{M}$ by Theorem \ref {2.1}.

(ii) $\Rightarrow$ (i). We prove by using induction on $n$. Let $n=
0$ and consider the isomorphism $\Hom_{R}(R/\fa, X)\cong
\Hom_{R}(R/\fa, \Gamma_{\fa}(X)).$ Since $\Hom_{R}(R/\fa, X)$ is in
$\mathcal{M}$, $\Hom_{R}(R/\fa, \Gamma_{\fa}(X))$ is in
$\mathcal{M}$. Thus $\Gamma_{\fa}(X)$ is in $\mathcal{M}.$

Now, suppose that $n> 0$ and that $n-1$ is settled. Since
$\Ext^{i}_{R}(R/\fa, X)$ is in $\mathcal{M}$ for all $i$, $0\leq
i\leq n-1$, $H^{i}_{\fa}(X)$ is in $\mathcal{M}$ for all $i$, $0
\leq i\leq n-1$ by the induction hypothesis. Now, by the above
proposition, $H^{n}_{\fa}(X)$ is in $\mathcal{M}$.
\end{proof}



In the next proposition, we study the membership of the local
cohomology modules of an $R$--module $X$ with respect to different
ideals in Melkersson subcategories which, among other things, shows
that $\cd_{\mathcal{M}}(\fb, X)\leq \cd_{\mathcal{M}}(\fa, X)+
\ara(\fb/\fa)$ where $\mathcal{M}$ is a Melkersson subcategory of
$\mathcal{C}(R)$ and $\fb$ is an ideal of $R$ contains $\fa$.

\begin{prop} \label {3.4} Let $X$ be an $R$--module and $\fb$ be an ideal of $R$
such that $\fa\subseteq \fb$. Assume also that $\mathcal{M}$ is a
Melkersson subcategory of $\mathcal{C}(R)$ and $n$ is a non-negative
integer such that $H^{i}_{\fa}(X)$ is in $\mathcal{M}$ for all $i$,
$0\leq i\leq n$ \emph{(}resp.  $i\geq n$\emph{)}. Then
$H^{i}_{\fb}(X)$ is in $\mathcal{M}$ for all $i$, $0\leq i\leq n$
\emph{(}resp. $i\geq n+\emph\ara(\fb/\fa)$\emph{)}.
\end{prop}

\begin{proof} Let $r= \ara(\fb/\fa)$. There exist $x_{1},...,x_{r}\in R$
such that $\sqrt{\fb}= \sqrt{\fa + (x_{1},...,x_{r})}$. We can, and
do, assume that $\fb= \fa+ \fc$ where $\fc= (x_{1},...,x_{r})$. By
\cite[Theorem 11.38]{Rot}, there is a Grothendieck spectral sequence
$$E^{p, q}_{2}:= H^{p}_{\fc}(H^{q}_{\fa}(X))
_{\stackrel{\Longrightarrow}{p}} H^{p +q}_{\fb}(X).$$ Assume that
$t$ is a non-negative integer such that $0\leq t\leq n$ (resp.
$t\geq n+ r$). For all $i$, $0\leq i\leq t$, $E_{\infty}^{t- i, i}=
E_{t+2}^{t-i, i}$ since $E_{j}^{t- i- j, i+ j- 1}= 0= E_{j}^{t- i+
j, i- j+ 1}$ for all $j\geq t+2$. Therefore $E_{\infty}^{t- i, i}$
is in $\mathcal{M}$ from the fact that $E_{t+ 2}^{t- i, i}$ is a
subquotient of $E_{2}^{t- i, i}= H^{t- i}_{\fc}(H^{i}_{\fa}(X))$
which belongs to $\mathcal{M}$ by assumption and Proposition \ref
{3.3}. There exists a finite filtration
$$0= \phi^{t+1}H^{t}\subseteq \phi^{t}H^{t}\subseteq \cdots\subseteq \phi^{1}H^{t}\subseteq \phi^{0}H^{t}= H^{t}_{\fb}(X)$$
such that $E_{\infty}^{t- i, i}= \phi^{t- i}H^{t}/\phi^{t- i+
1}H^{t}$ for all $i$, $0\leq i\leq t.$ Now the exact sequences
$$0\lo \phi^{t- i+ 1}H^{t}\lo \phi^{t- i}H^{t}\lo E_{\infty}^{t- i, i}\lo 0,$$
for all $i$, $0\leq i\leq t$, show that $H^{t}_{\fb}(X)$ is in
$\mathcal{M}$.
\end{proof}



Let $\fa$ be an ideal of $R$, $N$ a finite $\fa$--torsion module and
$s, t$ non-negative integers. In the following theorem, we find some
sufficient conditions for validity of the isomorphism
$\Ext^{s+t}_{R}(N, X)\cong \Ext^{s}_{R}(N, H^{t}_{\fa}(X))$ which
concerns to the case $\mathcal{S}= 0$.

\begin{thm} \label {3.5} Let $X$ be an $R$--module and $s, t$ be
non-negative integers such that
               \begin{itemize}
                   \item[(i)]{$\emph{\Ext}^{s+ t- i}_{R}(N, H^{i}_{\fa}(X))= 0$ for all $i$, $0\leq i< t$ or $t< i\leq s+ t$,}
                   \item[(ii)]{$\emph{\Ext}^{s+ t+ 1- i}_{R}(N, H^{i}_{\fa}(X))= 0$ for all $i$, $0\leq i< t,$ and}
                   \item[(iii)]{$\emph{\Ext}^{s+ t- 1- i}_{R}(N, H^{i}_{\fa}(X))= 0$ for all $i$, $t+ 1\leq i< s+ t.$}
               \end{itemize}
Then we have $\emph{\Ext}^{s}_{R}(N, H^{t}_{\fa}(X))\cong
\emph{\Ext}^{s+t}_{R}(N, X)$.
\end{thm}

\begin{proof} Consider the Grothendieck spectral sequence
$$E^{p, q}_{2}:= \Ext^{p}_{R}(N, H^{q}_{\frak{a}}(X))
_{\stackrel{\Longrightarrow}{p}} \Ext^{p +q}_{R}(N, X)$$ and, for
all $r\geq 2$, the exact sequences
$$0\rightarrow B_{r}^{s,t}\rightarrow Z_{r}^{s, t}\rightarrow
E_{r+1}^{s, t}\rightarrow 0 \textmd{ \ and \ } 0\rightarrow
Z_{r}^{s, t}\rightarrow E_{r}^{s, t}\rightarrow E_{r}^{s, t}/Z_{r}^{s,
t}\rightarrow 0$$
as we used in Theorem \ref{2.3}. Since $E_{2}^{s+r, t+1-r}= 0= E_{2}^{s-r,
t+r-1}$, $E_{r}^{s+r,
t+1-r}= 0= E_{r}^{s-r, t+r-1}$. Therefore $E_{r}^{s, t}/Z_{r}^{s, t}= 0= B_{r}^{s, t}$
which shows that $E^{s, t}_{r}= E^{s, t}_{r+ 1}$. Hence we have
$$\Ext^{s}_{R}(N, H^{t}_{\fa}(X))= E^{s, t}_{2}= E^{s, t}_{3}= \cdots= E^{s, t}_{s+t+1}= E^{s, t}_{s+t+2}= E^{s,t}_{\infty}.$$
There is a finite filtration
$$0= \phi^{s+t+1}H^{s+t}\subseteq \phi^{s+t}H^{s+t}\subseteq \cdots\subseteq \phi^{1}H^{s+t}\subseteq \phi^{0}H^{s+t}= \Ext^{s+t}_{R}(N, X)$$
such that $E^{s+t-j, j}_{\infty}=
\phi^{s+t-j}H^{s+t}/\phi^{s+t-j+1}H^{s+t}$ for all $j$, $0\leq j\leq
s+t.$ Note that for each $j$, $0\leq j\leq t- 1$ or $t+1\leq j\leq
s+ t$, by assumption (i), we have $E^{s+t-j, j}_{\infty}= 0$.
Therefore we get
$$0= \phi^{s+t+1}H^{s+t}= \phi^{s+t}H^{s+t}= \cdots= \phi^{s+2}H^{s+t}= \phi^{s+1}H^{s+t}$$
and
$$\phi^{s}H^{s+t}= \phi^{s-1}H^{s+t}= \cdots= \phi^{1}H^{s+t}= \phi^{0}H^{s+t}= \Ext^{s+t}_{R}(N, X).$$
Thus $\Ext^{s}_{R}(N, H^{t}_{\fa}(X))= E^{s, t}_{\infty}=
\phi^{s}H^{s+t}/\phi^{s+1}H^{s+t}= \Ext^{s+t}_{R}(N, X)$ as desired.
\end{proof}

The following corollaries are immediate applications of the above
theorem which give us some useful isomorphisms and equalities about
the extension functors and Bass numbers of local cohomology modules,
respectively.


\begin{cor}\label {3.6} {\rm(cf.} \cite [Corollary 4.2.(c)]{AM2}{\rm)}
Let $X$ be an $R$--module and $n$ be a non-negative integer. Then
the isomorphism $\emph{\Hom}_{R}(N,H^{n}_{\frak{a}}(X))\cong
\emph{\Ext}^{n}_{R}(N,X)$ holds in either of the following cases:
     \begin{itemize}
         \item[(i)]{$\emph{\Ext}^{j-i}_{R}(N, H^{i}_{\fa}(X))= 0$ for all $i, j$ with $0\leq i\leq n-1$ and $j= n, n+1;$}
          \item[(ii)]{$\emph{\Ext}^{i}_{R}(R/\fa, X)= 0$ for all $i$, $0\leq i\leq n-1.$}
     \end{itemize}
\end{cor}

\begin{proof}
(i) Apply Theorem \ref {3.5} with $s= 0$ and $t= n$.

(ii) By Proposition \ref {3.3}, $H^{i}_{\fa}(X)= 0$ for all $i$,
$0\leq i\leq n- 1$. Now, use case (i).
\end{proof}



\begin{cor} \label {3.7}  Let $X$ be an $R$--module and $n, m$ be  non-negative
integers such that $n\leq m$. Assume also that $H^{i}_{\fa}(X)= 0$
for all $i$, $i\neq n$ {\rm(}resp. $0\leq i\leq n-1$ or $n+1\leq
i\leq m${\rm)}. Then we have $\emph{\Ext}^{i}_{R}(N,
H^{n}_{\frak{a}}(X))\cong \emph{\Ext}^{i+ n}_{R}(N, X)$ for all $i$,
$i\geq 0$ {\rm(}resp. $0\leq i\leq m-n${\rm)}.
\end{cor}

\begin{proof}
For all $i$, $i\geq 0$ {\rm(}resp. $0\leq i\leq m-n${\rm)}, apply
Theorem \ref {3.5} with $s= i$ and $t= n$.
\end{proof}



\begin{cor} \label {3.8} {\rm(cf.} \cite [Proposition 3.1]{DY3}{\rm)}
Let $X$ be an $R$--module and $n$ be a non-negative integer such
that $H^{i}_{\fa}(X)= 0$ for all $i$, $i\neq n.$ Then we have
$\mu^{i}(\fp, H^{n}_{\fa}(X))= \mu^{i+n}(\fp, X)$ for all $i\geq 0$
and all $\fp\in V(\fa).$
\end{cor}

\begin{proof} Let $\fp\in V(\fa).$ By assumption, $H^{i}_{\fa R_{\fp}}(X_{\fp})= 0$ for all  $i$, $i\neq
n$; so that $\Ext^{i}_{R_{\fp}}(R_{\fp}/\fp R_{\fp},H^{n}_{\fa
R_{\fp}}(X_{\fp}))\cong \Ext^{i+n}_{R_{\fp}}(R_{\fp}/\fp
R_{\fp},X_{\fp})$ for all $i\geq 0$ by Corollary \ref {3.7}. Hence
$\mu^{i}(\fp, H^{n}_{\fa}(X))= \mu^{i+n}(\fp, X)$ for all $i\geq 0$.
\end{proof}



\begin{cor} \label {3.9} For an arbitrary $R$--module $X$, the following statements hold
true.
               \begin{enumerate}
                   \item[\emph{(i)}]{If $\emph{\cd}(\fa, X)= 0$, then $\emph{\Ext}^{i}_{R}(N, \Gamma_{\fa}(X))\cong
                   \emph{\Ext}^{i}_{R}(N, X)$ for all $i\geq 0$.}
                   \item[\emph{(ii)}]{If $\emph{\cd}(\fa, X)= 1$, then $\emph{\Ext}^{i}_{R}(N, H^{1}_{\fa}(X))\cong
                   \emph{\Ext}^{i+ 1}_{R}(N, X/\Gamma_{\fa}(X))$ for all $i\geq 0$.}
                   \item[\emph{(iii)}]{If $\emph{\cd}(\fa, X)= 2$, then $\emph{\Ext}^{i}_{R}(N, H^{2}_{\fa}(X))\cong
                   \emph{\Ext}^{i +2}_{R}(N, D_{\fa}(X))$ for all $i\geq 0$.}
               \end{enumerate}
\end{cor}

\begin{proof} By considering Corollary \ref {3.7}, this is similar
to that of Corollary \ref {2.9}.
\end{proof}


\section{Cofinite modules}

We first introduce the class of cofinite modules with respect to an
ideal and a Serre subcategory of the category of $R$-modules.



\begin{defn} \label {4.1} Let $\fa$ be an ideal of
$R$, $X$ be an $R$--module and $\mathcal{S}$ be a Serre subcategory
of $\mathcal{C}(R)$. We say that $X$ is {\it $\mathcal{S}$--cofinite
with respect to the ideal $\fa$} if $\emph{\Supp}_{R}(X)\subseteq
V(\fa)$ and $\emph{\Ext}^{i}_{R}(R/\fa, X)$ is in $\mathcal{S}$ for
all $i\geq 0.$ We will denote this concept by {\it $(\mathcal{S},
\fa)$--cofinite}.
\end{defn}

Note that when $\mathcal{S}$ is $\mathcal{C}_{f.g}(R)$ (resp.
$\mathcal{C}_{w.l}(R)$), $X$ is $(\mathcal{S}, \fa)$--cofinite
exactly when $X$ is $\fa$--cofinite (resp. $\fa$--weakly cofinite).



\begin{thm} \label {4.2}  Let $X$ be an $R$--module and $n$ be a non-negative integer
such that $H^{i}_{\fa}(X)$ is $(\mathcal{S}, \fa)$--cofinite for all
$i$, $i\neq n$. Then the following statements are equivalent.
               \begin{itemize}
                   \item[(i)]{$\emph{\Ext}^{i}_{R}(R/\fa, X)$ is in $\mathcal{S}$ for all $i\geq 0$.}
                   \item[(ii)]{$\emph{\Ext}^{i}_{R}(R/\fa, X)$ is in $\mathcal{S}$ for all $i\geq n$.}
                   \item[(iii)]{$H^{n}_{\fa}(X)$ is $(\mathcal{S}, \fa)$--cofinite.}
               \end{itemize}
\end{thm}

\begin{proof} (i) $\Rightarrow$ (ii). This is clear.

(ii) $\Rightarrow$ (iii).  For all $i\geq 0$, apply Theorem \ref
{2.3} with $N= R/\fa$, $s= i$ and $t= n$.

(iii) $\Rightarrow$ (i).  Apply Theorem \ref {2.1} with $N= R/\fa$.
\end{proof}



As an immediate result, the following corollary recovers and
improves \cite [Proposition 2.5]{MV}, \cite [Proposition 3.11]{Mel}
and \cite [Theorem 3.1]{DM2}.

\begin{cor} \label {4.3} {\rm(cf.} \cite [Proposition 2.5]{MV},
 \cite [Proposition 3.11]{Mel} and  \cite [Theorem 3.1]{DM2}{\rm)}
Let $X$ be an $R$--module and $n$ be a non-negative integer such
that $H^{i}_{\fa}(X)$ is $\fa$--cofinite {\rm(}resp. $\fa$--weakly
cofinite{\rm)} for all $i$, $i\neq n$. Then the following statements
are equivalent.
               \begin{itemize}
                   \item[(i)]{$\emph{\Ext}^{i}_{R}(R/\fa, X)$ is finite {\rm(}resp. weakly Laskerian{\rm)} for all $i\geq 0$.}
                   \item[(ii)]{$\emph{\Ext}^{i}_{R}(R/\fa,X)$ is finite {\rm(}resp. weakly Laskerian{\rm)} for all $i\geq n$.}
                   \item[(iii)]{$H^{n}_{\fa}(X)$ is $\fa$--cofinite {\rm(}resp. $\fa$--weakly cofinite{\rm)}.}
               \end{itemize}
\end{cor}



\begin{thm} \label {4.4} Suppose that $X$ is an
$R$--module and $n$ is a non-negative integer such that
               \begin{itemize}
                   \item[(i)]{$H^{i}_{\fa}(X)$ is $(\mathcal{S}, \fa)$--cofinite for all $i$, $0\leq i\leq n- 1$, and}
                   \item[(ii)]{$\emph{\Ext}^{1+n}_{R}(N, X)$ is in $\mathcal{S}$.}
               \end{itemize}
Then $\emph{\Ext}^{1}_{R}(N, H^{n}_{\fa}(X))$ is in $\mathcal{S}$.
\end{thm}

\begin{proof} Consider \cite [Proposition 3.4]{HV} and apply Theorem \ref
{2.3} with $s= 1$ and $t= n$.
\end{proof}



The following result is an application of the above theorem.

\begin{cor} \label {4.5} {\rm(cf.} \cite [Theorem A]{DY1} and \cite [Corollary 2.7]{DM1}{\rm)}
Let $X$ be an $R$--module and $n$ be a non-negative integer. Assume
also that
               \begin{itemize}
                   \item[(i)]{$H^{i}_{\fa}(X)$ is $\fa$--cofinite {\rm(}resp. $\fa$--weakly cofinite{\rm)} for all $i$, $0\leq i\leq n- 1$, and}
                   \item[(ii)]{$\emph{\Ext}^{1+n}_{R}(N, X)$ is finite {\rm(}resp. weakly Laskerian{\rm)}.}
               \end{itemize}
Then $\emph{\Ext}^{1}_{R}(N, H^{n}_{\fa}(X))$ is finite {\rm(}resp.
weakly Laskerian{\rm)}.
\end{cor}



\begin{thm} \label {4.6} Let $X$ be an $R$--module and
$n$ be a non-negative integer such that $\emph{\Ext}^{n+1}_{R}(N,
X)$ and $\emph{\Ext}^{n+2}_{R}(N, X)$ are in $\mathcal{S}$, and
$H^{i}_{\fa}(X)$ is $(\mathcal{S}, \fa)$--cofinite for all $i$,
$0\leq i< n.$ Then the following statements are equivalent.
              \begin{itemize}
                 \item[(i)]{$\emph{\Hom}_{R}(N, H^{n+1}_{\fa}(X))$ is in $\mathcal{S}$.}
                 \item[(ii)]{$\emph{\Ext}^{2}_{R}(N, H^{n}_{\fa}(X))$ is in $\mathcal{S}$.}
              \end{itemize}
\end{thm}

\begin{proof}
(i) $\Rightarrow$ (ii).  Consider \cite [Proposition 3.4]{HV} and
apply Theorem \ref {2.3} with $s= 2$ and $t= n$.

(ii) $\Rightarrow$ (i).  Again consider \cite [Proposition 3.4]{HV}
and apply Theorem \ref {2.3} with  $s= 0$ and $t= n+1$.
\end{proof}



Asadollahi and Schenzel proved that over local ring $(R, \fm)$, if
$X$ is a Cohen-Macaulay $R$-module and $t= \gr(\fa, X)$ then
$\Hom_{R}(R/\fa, H^{t+1}_{\fa}(X))$ is finite if and only if
$\Ext^{2}_{R}(R/\fa, H^{t}_{\fa}(X))$ is finite (see \cite [Theorem
1.2]{ASc}). Dibaei and Yassemi, in \cite{DY1}, generalized this
result with weaker assumptions on $R$ and $X$. As an immediate
consequence of Theorem \ref {4.6}, the following is a generalization
of \cite [Theorem B]{DY1}.

\begin{cor} \label {4.7} {\rm(cf.} \cite [Theorem B]{DY1}{\rm)} Let $X$ be an $R$--module and
$n$ be a non-negative integer. Assume also that
$\emph{\Ext}^{n+1}_{R}(N, X)$ and $\emph{\Ext}^{n+2}_{R}(N, X)$ are
finite {\rm(}resp. weakly Laskerian{\rm)}, and $H^{i}_{\fa}(X)$ is
$\fa$--cofinite {\rm(}resp. $\fa$--weakly cofinite{\rm)} for all
$i$, $0\leq i< n.$ Then the following statements are equivalent.
              \begin{itemize}
                   \item[(i)]{$\emph{\Hom}_{R}(N, H^{n+1}_{\fa}(X))$ is finite {\rm(}resp. weakly Laskerian{\rm)}.}
                   \item[(ii)]{$\emph{\Ext}^{2}_{R}(N, H^{n}_{\fa}(X))$ is finite {\rm(}resp. weakly Laskerian{\rm)}.}
               \end{itemize}
\end{cor}



In \cite [Proposition 2] {DM}, Delfino and Marley proved the Change
of ring principle for cofiniteness. In the following theorem, we
prove it for Serre cofiniteness. The proof is an adaption of the
proof of \cite [Proposition 2] {DM}.

\begin{thm} \label {4.8} Let $\phi: A\longrightarrow B$ be a
homomorphism between Noetherian rings such that $B$ is a finite
$A$--module, $\fa$ be an ideal of $A$ and $X$ be a $B$--module. Let
$\mathcal{S}$ and $\mathcal{T}$ be Serre subcategories of
$\mathcal{C}(A)$ and $\mathcal{C}(B)$, respectively. Assume also
that for any $B$--module $Y$, $Y$ is in $\mathcal{T}$ exactly when
$Y$ is in $\mathcal{S}$ \emph{(}as an $A$--module\emph{)}. Then $X$
is $(\mathcal{T}, \fa B)$--cofinite if and only if $X$ is
$(\mathcal{S}, \fa)$--cofinite \emph{(}as an $A$--module\emph{)}.
\end{thm}

\begin{proof} By \cite [Theorem 11.65] {Rot}, there is a Grothendieck spectral sequence
$$E^{p, q}_{2}:= \Ext^{p}_{B}(\Tor^{A}_{q}(B, A/\fa), X)
_{\stackrel{\Longrightarrow}{p}} \Ext^{p+ q}_{A}(A/\fa, X).$$

$(\Rightarrow )$. For all $p$ and $q$, by \cite [Proposition 3.4]
{HV}, $E^{p, q}_{2}$ is in $\mathcal{S}$. Therefore $E^{p,
q}_{\infty}$ belongs to $\mathcal{S}$ since $E^{p, q}_{\infty}=
E^{p, q}_{p+ q+ 2}$ and $E^{p, q}_{p+ q+ 2}$ is a subquotient of
$E^{p, q}_{2}$. Let $n$ be a non-negative integer. There exists a
finite filtration
$$0= \phi^{n+1}H^{n}\subseteq \phi^{n}H^{n}\subseteq \cdots
 \subseteq \phi^{1}H^{n}\subseteq \phi^{0}H^{n}= \Ext^{n}_{A}(A/\fa, X)$$
such that $E_{\infty}^{n-i ,i}= \phi^{n-i}H^{n}/\phi^{n-i+1}H^{n}$
for all $i$, $0\leq i\leq n$. Now, by the exact sequences
$$0\longrightarrow \phi^{n-i+1}H^{n}\longrightarrow \phi^{n-i}H^{n}\longrightarrow E_{\infty}^{n-i ,i}\longrightarrow 0,$$
for all $i$, $0\leq i\leq n$, $\Ext^{n}_{A}(A/\fa, X)$ is in
$\mathcal{S}$.

$(\Leftarrow )$. By using induction on $n$, we show that $E^{n,
0}_{2}= \Ext^{n}_{B}(B/\fa B, X)$ is in $\mathcal{T}$ for all $n\geq
0$. The case $n= 0$ is clear from the isomorphism $\Hom_{B}(B/\fa B,
X)\cong \Hom_{A}(A/\fa, X)$. Assume that $n> 0$ and that $E^{p,
0}_{2}$ is in $\mathcal{T}$ for all $p$, $0\leq p\leq n- 1$. For all
$r\geq 2$, we have $E^{n, 0}_{r+ 1}\cong E^{n, 0}_{r}/\im(E^{n- r,
r- 1}_{r}\lo E^{n, 0}_{r})$. Thus $E^{n, 0}_{r}$ is in $\mathcal{T}$
whenever $E^{n, 0}_{r+ 1}$ is in $\mathcal{T}$ because $E^{n- r, r-
1}_{r}$ is in $\mathcal{T}$ by the induction hypotheses and \cite
[Proposition 3.4] {HV}. Since $E^{n, 0}_{\infty}= E^{n, 0}_{n+ 2}$,
to complete the proof it is enough to show that $E^{n, 0}_{\infty}$
is in $\mathcal{T}$. By assumption, $\Ext^{n}_{A}(A/\fa, X)$ is in
$\mathcal{T}$ and hence $\phi^{n}H^{n}$ is in $\mathcal{T}$. That is
$E^{n, 0}_{\infty}$ belongs to $\mathcal{T}$ as desired.
\end{proof}



\begin{defn} \label {4.9} {\rm(}see \cite {Z}{\rm)} The $R$--module $X$ is a {\it minimax module}
if it has a finite submodule $X'$ such that $X/X'$ is Artinian.
\end{defn}

The class of minimax modules thus includes all finite and all
Artinian modules. Note that the category of minimax modules and the
category of $\fa$--cofinite minimax modules are two Serre
subcategories of the category of $R$--modules (see \cite [Corollary
4.4]{Mel}).



\begin{prop} \label {4.10} Let $X$ be an $R$--module and $n, m$ be non-negative integers
such that $n\leq m$. Assume also that
               \begin{itemize}
                   \item[(i)]{$H^{i}_{\fa}(X)$ is $\fa$--cofinite for all $i$, $0\leq i\leq n-1$,}
                   \item[(ii)]{$\emph{\Ext}^{i}_{R}(R/\fa, X)$ is finite for all $i$, $n\leq i\leq m$, and}
                   \item[(iii)]{$H^{i}_{\fa}(X)$ is minimax for all $i$, $n\leq i\leq m$.}
               \end{itemize}
Then $H^{i}_{\fa}(X)$ is $\fa$--cofinite for all $i$, $0\leq i\leq
m$.
\end{prop}

\begin{proof} Apply Theorem \ref {2.3} with $s= 0$ and $t= n$ for $N= R/\fa$ and $\mathcal{S}= \mathcal{C}_{f.g}(R).$
It shows that $\Hom_{R}(R/\fa, H^{n}_{\fa}(X))$ is finite. Thus
$H^{n}_{\fa}(X)$ is $\fa$--cofinite from \cite [Proposition
4.3]{Mel}.
\end{proof}



\begin{cor} \label {4.11} {\rm(cf.} \cite [Theorem 2.3]{BN}{\rm)}
Let $X$ be an $R$--module and $n$ be a non-negative integer such
that
               \begin{itemize}
                   \item[(i)]{$H^{i}_{\fa}(X)$ is minimax for all $i$, $0\leq i\leq n-1$, and}
                   \item[(ii)]{$\emph{\Ext}^{i}_{R}(R/\fa, X)$ is finite for all $i$, $0\leq i\leq n.$}
               \end{itemize}
Then $\emph{\Hom}_{R}(R/\fa, H^{n}_{\fa}(X))$ is finite.
\end{cor}

\begin{proof} By \cite [Proposition 4.3]{Mel}, $\Gamma_{\fa}(X)$ is
$\fa$--cofinite. Hence $H^{i}_{\fa}(X)$ is $\fa$--cofinite for all
$i$, $0\leq i\leq n-1$, from Proposition \ref {4.10}. Thus, by
Theorem \ref {2.3}, $\Hom_{R}(R/\fa, H^{n}_{\fa}(X))$ is finite.
\end{proof}



\begin{cor} \label {4.12} Suppose that $X$ is an $R$--module and that $n$ is a non-negative
integer. Then the following statements are equivalent.
               \begin{itemize}
                   \item[(i)]{$H^{i}_{\fa}(X)$ is Artinian $\fa$--cofinite for all $i$, $0\leq i\leq n$.}
                   \item[(ii)]{$\emph{\Ext}^{i}_{R}(R/\fa, X)$ has finite length for all $i$, $0\leq i\leq n.$}
               \end{itemize}
\end{cor}

\begin{proof}
(i) $\Rightarrow$ (ii). Let $0\leq t\leq n.$ Since
$\Ext^{t-i}_{R}(R/\fa, H^{i}_{\fa}(X))$ has finite length for all
$i$, $0\leq i\leq t,$ $\Ext^{t}_{R}(R/\fa, X)$ has also finite
length by Theorem \ref {2.1}.

(ii) $\Rightarrow$ (i). By Proposition \ref {3.3}, $H^{i}_{\fa}(X)$
is Artinian for all $i$, $0\leq i\leq n.$ Let $0\leq t\leq n$ and
consider Corollary \ref {4.11}. It shows that $\Hom_{R}(R/\fa,
H^{t}_{\fa}(X))$ is finite and so has finite length. Now, the
assertion follows from \cite [Proposition 4.3]{Mel}.
\end{proof}



\bibliographystyle{amsplain}

\end{document}